\documentclass[12pt,oneside,reqno]{amsart}
\usepackage{hyperref,amsmath,amssymb,amsthm}
\usepackage{mathrsfs}
\usepackage{hyperref}
\usepackage{physics}
\usepackage[normalem]{ulem}

\newtheorem{theorem}{Theorem}[section]
\newtheorem{corollary}[theorem]{Corollary}
\newtheorem{lemma}[theorem]{Lemma}
\newtheorem{proposition}[theorem]{Proposition}

\allowdisplaybreaks

\theoremstyle{definition}
\newtheorem{definition}[theorem]{Definition}
\newtheorem{example}[theorem]{Example}

\theoremstyle{remark}

\numberwithin{equation}{section}

\DeclareMathOperator{\g}{grad}

\begin{document}
\title{strict stability of calibrated cones}
\author{Bryan Dimler and Jooho Lee}

\address{Department of Mathematics\\
University of California, Irvine\\
CA, 92617}
\email{bdimler@uci.edu}

\address{Department of Mathematics\\
University of California, Irvine\\
CA, 92617}
\email{joohol3@uci.edu}

\begin{abstract}
    We study the strict stability of calibrated cones with an isolated singularity. Special Lagrangian cones are proved to be strictly stable when $n \geq 5$, and when $n = 4$ if the link is simply connected. All coassociative cones are shown to be strictly stable. In the complex case, we give non-strictly stable examples. 
\end{abstract}

\date{}

\maketitle

 \section{Introduction}

Let $C$ be an embedded $n$-dimensional minimal cone in $\mathbb{R}^m$ that is smooth away from an isolated singularity at the origin, let $\mathcal{A}$ be its corresponding area functional, and let $L$ be the stability operator on $C$. The cone $C$ is said to be stable if the second variation of $\mathcal{A}$ is nonnegative with respect to smooth normal variations which are compactly supported in $C$ (i.e. $\mathcal{A}^{\prime \prime}(0) \geq 0$). We say that $C$ is strictly stable if $L$ (identified with $-L$) is positive definite (see \eqref{sstable}). Thus, strict stability corresponds to a strong analytic condition giving a quantitative lower bound on the spectrum of $L$. 

In the context of calibrated cones, particularly special Lagrangian and coassociative cones, the notion of stability has taken on a different meaning and is important in the deformation theory (e.g. see \cite{Ha3, Jo2, L2, Ohn}). However, the notions of stability we use are classical and have played a fundamental role in the regularity theory for minimal hypersurfaces (\cite{B, CHS, HS, SS, Si}) since singular points of minimal submanifolds are conical at small scales (\cite{G, Sim}). In codimension one, many examples of (strictly) stable minimal cones have been discovered (\cite{CHS, Law, Li, O, Se, Si}). While many examples of calibrated (hence, stable) high codimension cones have been discovered (\cite{HL, Ha1, Ha2, Ha4, Jo1, Jo2, Mc}), to the authors' knowledge little is known regarding their strict stability.

In this paper, we prove the strict stability of special Lagrangian and coassociative calibrated cones with an isolated singularity. Precisely, we show that all special Lagrangian cones are strictly stable when $n \geq 5$, as well as when $n = 4$ provided the link of the cone is simply connected. Coassociative cones will be shown to always be strictly stable. We also provide examples of complex calibrated cones which are stable, but not strictly stable. 

\section*{Acknowledgments}
The authors would like to express their sincere gratitude to Rick Schoen and Connor Mooney for bringing this question to their attention and for several enlightening discussions on the subject of this paper. B. Dimler was partially supported by NSF grant DMS-2143668 and C. Mooney's Sloan and UCI Chancellor's Fellowships. J. Lee was partially supported by NSF Grant DMS-2005431. Part of this work was done while the authors were visiting the Simons Laufer Mathematical Sciences Institute with funding from NSF Grant DMS-1928930. The authors gratefully acknowledge the support of the SLMath.

\section{Preliminaries}

\subsection{Definitions}
Let $M$ be a smooth embedded $n$-dimensional Riemannian submanifold of $\mathbb{R}^{m}$ (with or without boundary) with the induced metric $\langle \cdot, \cdot \rangle$, where $m \geq n+1$. We will denote the normal bundle and tangent bundles of $M$ by $NM$ and $TM$. The space of smooth normal vector fields will be denoted by $\Gamma(NM)$, and we will write $\Gamma_c(NM)$ for the space of smooth normal vector fields with compact support in $M$. The space $\Gamma(TM)$ will represent the space of smooth tangent vector fields.

 We will denote the connections on $TM$ and $NM$ by $\nabla$ and $\nabla^\perp$, respectively. Let $A$ be the \emph{second fundamental form} of $M$. The \emph{mean curvature vector} $H$ on $M$ is the trace of $A$, and $M$ is a \emph{minimal submanifold} of $\mathbb{R}^m$ if $H \equiv 0$. We will write $|A|^2$ for length squared of the second fundamental form. The \emph{normal Laplacian} $\Delta^\perp$ on $\Gamma(NM)$ is given in a smooth local orthonormal frame $E_1,\ldots, E_n$ for $TM$ by
\begin{equation*}
    \Delta^\perp V = \sum_{i = 1}^n \big(\nabla^\perp_{E_i} \nabla^\perp_{E_i} V - \nabla^\perp_{\nabla_{E_i}E_i}V\big), \label{nlaplace}
\end{equation*}
and $L^2(NM)$ will denote the Hilbert space of Borel measurable sections of $NM$ that are square integrable with respect to the $n$-dimensional Hausdorff measure $\mathscr{H}^n$ on $M$.

\subsection{Stable Minimal Cones}
We will work with cones over smooth compact embedded submanifolds of $\mathbb{S}^{m-1}$. Let $\Sigma$ be a smooth compact embedded ($n-1$)-dimensional submanifold of $\mathbb{S}^{m-1}$ and let $C$ be the cone over $\Sigma$ with vertex at the origin. Set $r := |x|$ and $\sigma := |x|^{-1}x \in \mathbb{S}^{m-1}$. We will identify $C$ with  $C \setminus \{0\}$:
\begin{equation}
      C := \{r\sigma : \sigma \in \Sigma, \text{ } r \in (0,\infty)\}. \label{cone1}
\end{equation}
The cone $C$ is an $n$-dimensional embedded submanifold of $\mathbb{R}^m$ that has an isolated singularity at the origin, provided $\Sigma$ is not the totally geodesic $(n-1)$-sphere embedded in $\mathbb{S}^{m-1}$. In the latter case, $C$ is a piece of an $n$-dimensional plane in $\mathbb{R}^{m}$ passing through the origin. The cone $C$ is a minimal submanifold of $\mathbb{R}^m$ if and only if $\Sigma$ is a minimal submanifold of $\mathbb{S}^{m-1}$ (\cite{CHS, Si}).

Suppose $C$ is minimal in $\mathbb{R}^m$. If $V \in \Gamma_c(NC)$, then the \emph{second variation of area} if we vary in the direction $V$ is defined by the quadratic form
\begin{equation}
    Q(V,V) := \int_C \big(|\nabla^\perp V|^2 - \langle V, \tilde{A}(V) \rangle\big) = - \int_C \langle V ,\Delta^\perp V + \tilde{A}(V)\rangle, \label{jacobi}
\end{equation}
where the second equality follows from integration by parts and $\tilde{A}$ is \emph{Simons' operator} on $NC$. Let $E_1,\ldots, E_n$ be a smooth local orthonormal frame for $TC$. Then 
\begin{equation*}
    \tilde{A}(V) = \sum_{i,j = 1}^n\langle A(E_i,E_j) , V\rangle A(E_i,E_j).
\end{equation*}
From the expression above, it is easy to see that $\tilde{A}$ is symmetric with respect to $\langle \cdot, \cdot \rangle$ at each point $p \in C$. Set $L := \Delta^\perp + \tilde{A}$. The operator $L: \Gamma(NC) \rightarrow \Gamma(NC)$ is called the \emph{stability operator} (i.e. \emph{Jacobi operator}) on $C$. A section $V \in \Gamma(NC)$ is called a \emph{Jacobi field} if $LV = 0$.

For a given minimal cone $C$, define $C_1$ to be the intersection of $C$ with the closed unit ball (i.e. \eqref{cone1} with $r \in (0,1]$). The cone $C$ is said to be \emph{stable} if $Q(V,V) \geq 0$ for every $V \in \Gamma_c(NC_1)$ that vanishes on $\Sigma$, where the integral in \eqref{jacobi} is taken over $C_1$. We say that $C$ is \emph{strictly stable} if 
\begin{equation}
    \inf\Bigg\{-\frac{\int_{C_1} \langle V, LV \rangle}{\int_{C_1} |V|^2 r^{-2}} : V \in \Gamma_c(NC_1), \text{ } V \not\equiv 0, \text{ and } V = 0 \text{ on } \Sigma \Bigg\} > 0. \label{sstable}
\end{equation}
The weighted Rayleigh quotient \eqref{sstable} is chosen to preserve scaling. 

\subsection{Characterization of Stability}
Let $C$ be an $n$-dimensional embedded minimal cone in $\mathbb{R}^m$ with smooth link $\Sigma$. We extend the spectral characterization of stability for codimension one minimal cones to high codimension minimal cones in $\mathbb{R}^m$ (see \cite{CHS, Si, Sm}) following the presentation in \cite{Si}.

Fix $\sigma \in \Sigma$ and let $e_1,\ldots,e_{n-1}$ be an orthonormal basis for $T_\sigma \Sigma$. Extend the $e_i$ to a smooth local orthonormal frame $E_1,\ldots,E_{n-1}$ for $T\Sigma$ near $\sigma$, chosen so that the $E_i$ are covariant constant at $\sigma$ with respect to the connection on $T\Sigma$. Let $E_r := \pdv{r}$ be the unit radial vector in $\mathbb{R}^{m}$ and choose $N_1,\ldots, N_{m-n} \in \Gamma(N\Sigma)$ such that $N_1,\ldots, N_{m-n}$ is a local orthonormal frame for $N\Sigma$ near $\sigma$. Extend $E_1,\ldots,E_{n-1}, N_1,\ldots, N_{m-n}$ to $TC$ and $NC$ by parallel translation along the radial paths from points on $\Sigma$ through the origin. In the chosen coordinates,
\begin{equation}
    A_C(E_i,E_j) = r^{-1}A_\Sigma(E_i,E_j) \text{ for each } i,j = 1,\ldots, n-1, \label{scale}
\end{equation}
where $A_C$ and $A_\Sigma$ are the second fundamental forms for $C$ and $\Sigma$, respectively. Using that $\Sigma$ is minimal in $\mathbb{S}^{m-1}$ and the definition of the $E_i$ via parallel transport, we obtain the polar coordinate expression
\begin{equation}
    L = \pdv[2]{r} + \frac{n-1}{r} \pdv{r}  +\frac{1}{r^2} \big( \Delta_\Sigma^\perp +  \tilde{A}_\Sigma\big). \label{pstab}
\end{equation}
Note that, by the Cauchy-Schwarz inequality, for any $V\in \Gamma(NC)$
\begin{equation}
    |\tilde{A}_C(V)| \leq r^{-2}|A_\Sigma|^2|V| \label{simonsbdd}
\end{equation}
and $|A_\Sigma|^2$ is a bounded function.

For any $\epsilon \in (0,1)$, we denote by $C_{\epsilon,1}$ the truncated cone
\begin{equation*}
    C_{\epsilon,1} := \{r\sigma : \sigma \in \Sigma, \text{ } r \in [\epsilon,1]\}
\end{equation*}
and define the weighted $L^2$ space $L_r^2(NC_{\epsilon,1})$ by
$$
    L_r^2(NC_{\epsilon,1}) := L^2(NC_{\epsilon,1};r^{-2}d\mathscr{H}^n).
$$
The truncated cone $C_{\epsilon,1}$ is a smooth embedded compact minimal submanifold of $\mathbb{R}^m$ with boundary for each $\epsilon$. Define $r^2 L: \Gamma(NC_{\epsilon,1}) \rightarrow \Gamma(NC_{\epsilon,1})$. Then $r^2 L$ has smooth coefficients which are uniformly bounded independent of $\epsilon$ due to \eqref{scale} and \eqref{simonsbdd}. Since $L: \Gamma(NC_{\epsilon,1}) \rightarrow \Gamma(NC_{\epsilon,1})$ is strongly elliptic and $L^2(NC_{\epsilon,1})$ self-adjoint when restricted to sections vanishing on $\partial C_{\epsilon,1}$ (\cite{Si}, Propositions 1.2.2 and 1.2.3), $r^2 L$ is a strongly elliptic $L_r^2(NC_{\epsilon,1})$ self-adjoint operator. Therefore, the spectrum of $r^2L$ consists of a discrete set of eigenvalues which are bounded from below:
$$
    \lambda_1 \leq \lambda_2 \leq \lambda_3 \leq \cdots
$$
where $\lambda_i := \lambda_i(\epsilon)$ and 
$$
    r^2L V_{i} + \lambda_i V_{i} = 0
$$
has some nontrivial solution $V_{i} \in \Gamma(NC_{\epsilon,1})$ ($V_i := V_i(\epsilon))$ satisfying $V_i = 0$ on $\partial C_{\epsilon,1}$. The normal sections $V_{i}$ are the eigensections for $r^2L$ corresponding to $\lambda_i$, and can be chosen to be an orthonormal basis for $L_r^2(NC_{\epsilon,1})$.

Set 
\begin{equation}
    d_0(C_1) := \inf_{\epsilon \in (0,1)} \lambda_1(\epsilon) \label{dnot}
\end{equation}
and let $L_\Sigma := \Delta_\Sigma^\perp + \tilde{A}_\Sigma$. As before, $L_\Sigma: \Gamma(N\Sigma) \rightarrow \Gamma(N\Sigma)$ is a $L^2(N\Sigma)$ self-adjoint strongly elliptic operator so there is a discrete set of eigenvalues for $L_\Sigma$ increasing to $+\infty$
$$
    \mu_1 \leq \mu_2 \leq \mu_3 \leq \cdots
$$
with corresponding orthonormal basis of smooth eigensections for $L^2(N\Sigma)$
$$
    V_1, V_2, V_3 \ldots, \text{ where } V_j \in \Gamma(N\Sigma) \text{ for each } j \in \mathbb{N}.
$$
We prove a formula for $d_0(C_1)$ in terms of the dimension $n$ of $C$ and $\mu_1$, and show that the value of $d_0(C_1)$ determines the stability of $C$. 

Define the operator $T$ on $C^\infty([0,1])$ by
$$
    T := r^2 \pdv[2]{r} + (n-1)r \pdv{r}.
$$
Diagonalizing $T$ in $C_{0}^{\infty}([\epsilon, 1])$ as in \cite{Si} for each $\epsilon \in (0,1)$, we obtain a sequence of eigenfunctions $\phi_{i} \in C_{0}^{\infty}([\epsilon, 1])$ with corresponding eigenvalues
$$
    \gamma_i(\epsilon) := \frac{(n-2)^2}{4} + \Big(\frac{i\pi}{\log \epsilon}\Big)^2.
$$
We will write $\gamma_i$ for simplicity. After scaling, the $\phi_{i}$ are orthonormal with respect to the weighted $L^2$ inner product 
$$
    (\phi, \psi ) := \int_\epsilon^1 \phi \cdot \psi r^{n - 3} \, dr.
$$
Combining the $L^2$ expansions for $L_\Sigma$ and $T$ shows that any $V \in \Gamma(NC_{\epsilon,1})$ satisfying $V = 0$ on $\partial C_{\epsilon,1}$ has a unique eigensection expansion
\begin{equation*}
    V(r\sigma) = \sum_{i,j= 1}^\infty a^{ij} \phi_{i}(r) V_j(\sigma) 
\end{equation*}
for constants $a^{ij}:= a^{ij}(\epsilon)$.

The computations above lead to a characterization of stability for minimal cones with an isolated singularity. This result was proved in \cite{CHS, Si} for minimal hypercones. In \cite{CHS}, the authors also state that their proof works in higher codimension. For completeness, we briefly sketch the proof in high codimension following the arguments in Section 6 of \cite{Si}. 

\begin{proposition}\label{stabcondition}
    Suppose $C$ is an $n$-dimensional minimal cone in $\mathbb{R}^m$. 
    \begin{enumerate}
        \item $C$ is stable (strictly stable) if and only if $d_0(C_1) \geq 0$ $(>0)$.
        \item We have
        \begin{equation}
            d_0(C_1) = \frac{(n-2)^2}{4} + \mu_1. 
        \end{equation}
    \end{enumerate}
\end{proposition}
\begin{proof}
    We first prove (2). Suppose $V \in \Gamma(NC_{\epsilon,1})$ vanishes on $\partial C_{\epsilon,1}$ and write 
    $$
         V(r\sigma) = \sum_{i,j} a^{ij} \phi_{i}(r) V_j(\sigma).
    $$
    Using the coarea formula and orthogonality, we obtain 
    \begin{align*}
        -\int_{C_{\epsilon,1}} \langle V, r^2L V \rangle r^{-2} \, d \mathscr{H}^n &= -\int_{\epsilon}^1 \int_\Sigma \langle r^{n-3}a^{ij} \phi_{i}V_j,  a^{kl}\phi_{k} L_\Sigma V_l + a^{kl} T \phi_{k} V_l \rangle \, dr d \mathscr{H}^{n-1} \\
         &= \int_{\epsilon}^1 \int_\Sigma \langle r^{n-3}a^{ij} \phi_{i}V_j,  a^{kl}(\gamma_k + \mu_l)\phi_{k}V_l \rangle \, dr d \mathscr{H}^{n-1} \\
         &=(a^{ij})^2(\gamma_i + \mu_j) \Big( \int_\epsilon^1 \phi_{i}^2 r^{n-3} \, dr\Big) \Big(\int_\Sigma |V_j|^2 \, d\mathscr{H}^{n-1}\Big) \\
         &\geq \sum_{i,j} (a^{ij})^2(\gamma_1 + \mu_1) \\
         &= (\gamma_1 + \mu_1) \norm{V}_{L_r^2}^2.
    \end{align*}
    It follows that $\lambda_1 \geq \gamma_1 + \mu_1$. Since equality is obtained when $V = \phi_{1} V_1$, we have $\lambda_1(\epsilon) = \gamma_1(\epsilon) + \mu_1$ for each $\epsilon \in (0,1)$. Taking the infimum over $\epsilon \in (0,1)$ proves (2). Now, if $V \in \Gamma_c(NC_1)$ vanishes on $\Sigma$, then the support of $V$ is contained in $C_{\epsilon,1}$ for some $\epsilon \in (0,1)$ so that $V$ vanishes on $\partial C_{\epsilon,1}$. Noting that the left-hand side above is equivalent to $Q(V,V)$ by cancelation and that equality is achieved in the inequality, we conclude (1). 
\end{proof}
\begin{proposition}\label{strict}
    Suppose that $C$ is a stable $n$-dimensional minimal cone in $\mathbb{R}^{m}$. Then $C$ is not strictly stable if and only if $C$ has a Jacobi field with homogeneity $(2-n)/2$.
\end{proposition}

\begin{proof}
    This follows from Proposition \ref{stabcondition} and the standard separation of variables.
\end{proof}

We close this section with an example and a discussion of applications. Aside from the results in the present paper, the authors are unaware of nontrivial examples of strictly stable minimal cones with high codimension. In codimension one, the canonical examples are \emph{Lawson's cones} (\cite{Law}). See \cite{CHS,Sm} as well.

\begin{example}[Lawson's Cones] \label{example1} 
Let $S_{r}^k$ denote the $k$-sphere of radius $r$ in $\mathbb{R}^{n+1}$ with center at the origin. If the link $\Sigma$ is given by $\Sigma := S_{r_1}^k \times S_{r_2}^l$, where $k + l = n-1$ and 
\begin{equation*}
        r_1 = \sqrt{\frac{(k+1)}{(k + l + 2)}}, \text{ } r_2 = \sqrt{\frac{(l + 1)}{(k + l + 2)}},
\end{equation*}
then $\Sigma$ is a minimal hypersurface of $\mathbb{S}^n$ and $|A_\Sigma|^2 = n-1$ implying $\mu_1 = 1-n$. Thus, the cone $C$ with link $\Sigma$ is strictly stable if and only if $n \geq 7$. When $n = 7$, $k = l = 3$, and $r_1 = r_2 = \frac{\sqrt{2}}{2}$, we get \emph{Simons' cone}. Simons' cone was crucial in resolving the Bernstein problem (\cite{B, G, Si}).
\end{example} 

Strict stability has been used to produce examples of nontrivial singular stable minimal hypersurfaces (\cite{CHS, Sm, Sm1}). Of particular interest to the authors is the work of Smale in \cite{Sm}. Smale constructed the first examples of (stable) minimal hypersurfaces with many isolated singularities by solving a fixed point problem for $L$ in order to glue strictly stable minimal hypercones. The strict stability condition was needed to obtain the $L^2$ and Schauder estimates necessary for solving the fixed point problem, as well as to ensure the singularities were preserved and that the constructed hypersurfaces were stable. In \cite{D2}, Smale's constructions were recently extended to high codimension. In fact, Theorem \ref{coassoc} and Example \ref{LOcone} together with \cite{D2} imply the existence of four dimensional Lipschitz minimal graphs in $\mathbb{R}^7$ having any finite number of isolated singularities.

\section{Calibrations}
We briefly define calibrations and calibrated submanifolds in $\mathbb{R}^m$. For a full discussion, see the original work of Harvey-Lawson in \cite{HL}.

\begin{definition}
\begin{enumerate}
        \item An $n$-form $\omega$ on $\mathbb{R}^m$ is a \emph{calibration} if $\omega$ is closed and for each $p \in \mathbb{R}^m$ we have $|\omega(e_1, \ldots, e_n)| \leq 1$ for all orthogonal unit vectors $e_i \in T_p\mathbb{R}^m$.
    \item Let $\omega$ be a calibration $n$-form on $\mathbb{R}^m$. An oriented $n$-dimensional submanifold $M$ of $\mathbb{R}^m$ is \emph{calibrated by $\omega$} if $\omega$ restricts to the volume form on $M$; i.e. for all $p \in M$ we have $\omega(e_1, \ldots, e_n) = 1$ for any oriented orthonormal basis $e_1,\ldots, e_n$ for $T_pM$.
\end{enumerate}
\end{definition}
Calibrated submanifolds are area-minimizing in their homology class (see Theorem II.4.2 in \cite{HL}); hence, are stable. We focus on complex, special Lagrangian, and coassociative calibrated cones.

\section{Complex Cones}
        In this section, we discuss complex cones. For additional background, we refer to \cite{HL}. Let $\mathbb{C}^n$ have the standard coordinates $z_j=x_j+iy_j$, for $j=1, \cdots, n$. We consider the standard complex structure $J$ and the K{\"a}hler form $\omega$. In the standard coordinates,
        \begin{align*}
            J\left(\frac{\partial}{\partial x_j} \right)=\frac{\partial}{\partial y_j}, \quad J\left(\frac{\partial}{\partial y_j} \right)=-\frac{\partial}{\partial x_j}\\
            \omega = \sum_{j=1}^{m} x_j \wedge y_j = \frac{i}{2} \sum_{j=1}^{m} dz_j \wedge d\bar{z}_j.
        \end{align*}
        Then, for each positive integer $k$, $\omega^{k}/k!$ is a calibration. 

        \begin{theorem}[\cite{HL}]
        An oriented $2k$-dimensional submanifold $M$ in $\mathbb{C}^{n}$ is a complex submanifold
        if and only if it is calibrated by $\omega^{k}/k!$ .
        \end{theorem}

         Let $f(z)$ be a homogeneous complex polynomial of degree $n$ defined on $\mathbb{C}^{n+1}$. For example, 
         \begin{equation*}
             f(z)=\sum_{j=1}^{n+1} (z_{j})^{n}.
         \end{equation*}
         Then, the zero set $f^{-1}(0)$ is a complex cone of real dimension $2n$. In particular, the cone has an isolated singularity if its complex exterior derivative $\partial f$ vanishes only at $0$, where 
         \begin{equation*}
             \partial f = \frac{\partial f}{\partial z_{j}} dz_{j}.
         \end{equation*}

\begin{theorem}
    Let $f(z)$ be a homogeneous polynomial of degree $n$ defined on $\mathbb{C}^{n+1}$. If $\partial f$ vanishes only at $0$, then the cone $C=f^{-1}(0)$ has a Jacobi field with homogeneity $(1-n)$. Thus, $C$ is not strictly stable.
\end{theorem}

\begin{proof}
    Let $f=u+iv$. Since $f$ is holomorphic, we have $\g v=J(\g u)$. For each $p\in C$, we consider an integral curve $c(t)$ of $|\g u|^{-2}\g u$ starting at $p$. Then, we have
    \begin{align*}
        (u \circ c)'(t) & = \g u \cdot |\g u|^{-2}\g u = 1 \\
        (v \circ c)'(t) & = \g v \cdot |\g u|^{-2}\g u = 0.
    \end{align*}
    Therefore, $c(t)\in f^{-1}(t)$. This implies that a family of complex submanifolds $f^{-1}(t)$ is a normal variation of the cone $C$. Therefore, $|\g u|^{-2}\g u$ is a Jacobi field on $C \setminus \{0\}$ with homogeneity $1-n=(2-2n)/2$.
\end{proof}

\section{Special Lagrangian Cones}
Here, we prove the strict stability of special Lagrangian cones. We refer to \cite{HL} for the basic theory and examples of special Lagrangian submanifolds and cones, and refer to \cite{Ha1, Ha3, Ha2, Ha4, Jo1, Jo2, Mc} for further examples of special Lagrangian cones and their constructions. Let $\mathbb{C}^n$ have the standard coordinates $z_j=x_j+iy_j$, for $j=1, \cdots, n$. We consider the standard complex structure $J$, K{\"a}hler form $\omega$, and holomorphic volume form $\Omega$. In the standard coordinates,
        \begin{align*}
            J\left(\frac{\partial}{\partial x_j} \right)=\frac{\partial}{\partial y_j}, \quad J\left(\frac{\partial}{\partial y_j} \right)=-\frac{\partial}{\partial x_j}\\
            \omega = \sum_{j=1}^{m} x_j \wedge y_j = \frac{i}{2} \sum_{j=1}^{n} dz_j \wedge d\bar{z}_j \\
            \Omega = dz_{1} \wedge \cdots \wedge dz_{n}.            
        \end{align*}

For each real number $\theta$, the real part $Re(e^{-i\theta}\Omega)$ is a calibration. 

\begin{definition}
    An oriented $n$-dimensional submanifold $L$ in $\mathbb{C}^{n}$ is a special Lagrangian submanifold (with angle $e^{i\theta}$) if $L$ is calibrated by $Re(e^{-i\theta}\Omega)$.
\end{definition}

First, we recall the property of closed and co-closed 1-forms on a cone.

\begin{theorem}{(\cite{FHN}, Appendix A)}
    Let $C^n$ be a cone in $\mathbb{R}^{m}$. Consider a homogeneous 1-form $\alpha$ of homogeneity $\lambda$:
    \begin{equation*}
        \alpha=r^{\lambda}\eta dr+r^{\lambda+1} \omega,
    \end{equation*}
    where $\eta$ is a function on $\Sigma$ and $\omega$ is a 1-form on $\Sigma$.
    If $\alpha$ is closed and co-closed, then

    \begin{align}
        d\eta-(\lambda+1)\omega&=0 \label{forms} \\
        d\omega&=0 \label{forms2} \\                
        \delta \omega -(\lambda+n-1)\eta&=0. \label{forms3}
    \end{align}
    In particular, we have
    \begin{equation}\label{hodge}
            \Delta_{\Sigma} \eta = (\lambda+1)(\lambda+n-1) \eta, \qquad
            \Delta_{\Sigma} \omega = (\lambda+1)(\lambda+n-1) \omega,
    \end{equation} where $\Delta_{\Sigma}$ is the Hodge Laplacian on $\Sigma$.
\end{theorem}

As an immediate corollary we obtain: 

\begin{corollary}\label{coro}
    Let $C$ be a special Lagrangian cone in $\mathbb{C}^{n}$. Let $\alpha$ be a 1-form of homogeneity $\lambda=(2-n)/2$ on $C$.    
    \begin{enumerate}
        \item If $n=4$ and the link $\Sigma$ is simply connected, then there is no closed and co-closed 1-form $\alpha$ on $C$ with homogeneity $\lambda=-1$.
        \item If $n>4$, then there is no closed and co-closed 1-form $\alpha$ with homogeneity $\lambda=(2-n)/2$.
    \end{enumerate}
\end{corollary}

\begin{proof}
    In case $(1)$, \eqref{forms} implies $\eta$ is a constant. Since $\Sigma$ is simply connected, \eqref{forms2} implies $\omega=df$. By \eqref{forms3}, $\Delta f$ is constant. This is possible only when $\eta=0$ and $f$ is constant. Thus, $\alpha=0$. 
    In case $(2)$, \eqref{hodge} implies $(\lambda+1)(\lambda+n-1)$ is a negative eigenvalue of the Hodge Laplacian. Thus, $\alpha=0$.  
\end{proof}

\begin{theorem}\label{Lthm}
Let $C$ be a special Lagrangian cone in $\mathbb{C}^{n}$. 
    \begin{enumerate}
        \item If $n=4$ and the link $\Sigma$ is simply connected, then $C$ is strictly stable.
        \item If $n>4$, then $C$ is strictly stable.        
    \end{enumerate}

\end{theorem}

    \begin{proof}
        Suppose that $C$ is not strictly stable. By Proposition \ref{strict}, there exists a Jacobi field $V$ of homogeneity $(2-n)/2$ satisfying $|V|^{2}\le K r^{2-n}$ for some constant $K>0$. We consider a logarithmic cutoff function $\varphi$
        \begin{equation*}
            \varphi(r)=   \begin{cases} 
                          0 &  (0 \le r\leq e^{-2N}) \\
                          2+\frac{\log r}{N} &  (e^{-2N}\leq r\le e^{-N}) \\
                          1 &  (e^{-N} \le r\leq e^{N}) \\ 
                          2-\frac{\log r}{N} &  (e^{N}  \le r\leq e^{2N}) \\
                          0 &  (e^{2N} \le r <  \infty).
                        \end{cases}
        \end{equation*}

        Next, we compute the second variation of the cone $C$ in the direction of the normal vector field $\varphi V$. Since $V$ is a Jacobi field, we have

        \begin{align*}
          Q(\varphi V,\varphi V)& =\int_{C} \big(|\nabla^\perp (\varphi V)|^2 - \langle (\varphi V), \tilde{A}(\varphi V) \rangle \\
                                & =\int_{C} |\nabla \varphi|^{2} |V|^{2}+\frac{1}{2} \langle \nabla \varphi^{2}, \nabla |V|^{2} \rangle + \varphi^2|\nabla^\perp V|^2 - \varphi^2 \langle V, \tilde{A}(V) \rangle \\
                                & =\int_{C} |\nabla \varphi|^{2} |V|^{2}-\frac{1}{2} \varphi^{2} \Delta |V|^{2} + \varphi^2|\nabla^\perp V|^2 - \varphi^2 \langle V, \tilde{A}(V) \rangle \\
                                & =\int_{C} |\nabla \varphi|^{2} |V|^{2} 
                                - \varphi^2 \langle V, \Delta^{\perp} V \rangle - \varphi^2 \langle V, \tilde{A}(V) \rangle \\
                                & =\int_{C} |\nabla \varphi|^{2} |V|^{2}.                                   
        \end{align*}
                
        By the coarea formula, we have
        
        \begin{align*}
            \int_{C} |\nabla \varphi|^{2} |V|^{2} & = \int_{C\cap \{ e^{-2N} < |x| < e^{-N}\} } |\nabla \varphi|^{2} |V|^{2} + \int_{C\cap \{ e^{N} < |x| < e^{2N}\} } |\nabla \varphi|^{2} |V|^{2} \\
            & \le \frac{K}{N^2} \left ( \int_{e^{-2N}}^{e^{-N}} r^{-n} \int_{r \Sigma} 1 \, d \mathscr{H}^{n-1} dr
            +
            \int_{e^{N}}^{e^{2N}} r^{-n} \int_{r \Sigma} 1 \, d \mathscr{H}^{n-1} dr \right ) \\
            &= \frac{K}{N^2} \left ( N \mathscr{H}^{n-1}(\Sigma) + N \mathscr{H}^{n-1}(\Sigma) \right ) \\
            &= \frac{2K}{N} \mathscr{H}^{n-1}(\Sigma).
        \end{align*}

        On the other hand, if we let $\alpha = (JV)^{\flat}$, then the second variation formula for special Lagrangian submanifolds (\cite{Oh}, Proposition 4.1; \cite{M}, Theorem 3.13) implies that 
        \begin{align*}
            Q(\varphi V,\varphi V) & = \int_{C} |d(\varphi \alpha)|^2 + |\delta(\varphi \alpha)|^2 \\
                                   & \ge \int_{C\cap \{ e^{-N} < |x| < e^{N}\}} |d \alpha|^2 + |\delta \alpha|^2. 
        \end{align*}
        Letting $N\to \infty$, by Lebesgue's monotone convergence theorem, we have $d \alpha=0$ and $\delta \alpha=0$. This contradicts Corollary \ref{coro}.
    \end{proof}

\section{Coassociative Cones}

We prove that all coassociative cones in $\mathbb{R}^7$ are strictly stable. For more on coassociative submanifolds and cones, we refer the reader to \cite{HL} and the works of Lotay (e.g. \cite{L3, L2, L}). Let $x^1, \ldots, x^7$ be coordinates on $\mathbb{R}^7$ with the Euclidean metric. Write $dx^{ij\ldots k}$ for the exterior form 
$$
    dx^i \wedge dx^j \wedge \cdots \wedge dx^k \text{ on } \mathbb{R}^7.
$$
and define the \emph{associative $3$-form} $\omega_0$ on $\mathbb{R}^7$ by
\begin{equation*}
    \omega_0 := dx^{567} + dx^5(dx^{12} - dx^{34}) + dx^6(dx^{13} - dx^{24}) + dx^7(dx^{14} - dx^{23}). \label{not}
\end{equation*}
The \emph{coassociative $4$-form} $\tilde{\omega}_0$ is then defined as the Hodge dual of $\omega_0$: $\tilde{\omega}_0 = *\omega_0$. In coordinates, 
\begin{equation*}
    \tilde{\omega}_0 := dx^{1234} - dx^{67}(dx^{12}- dx^{34})+ dx^{57}(dx^{13} + dx^{24}) - dx^{56}(dx^{14} - dx^{23}). \label{notstar}
\end{equation*}

\begin{definition}
    An oriented 4-dimensional submanifold $M$ in $\mathbb{R}^7$ is said to be \emph{coassociative} if it is calibrated by $\tilde{\omega}_0$. Equivalently, $M$ is coassociative if and only if $\omega_0 \lvert_M = 0$.
\end{definition}
The following example is due to Lawson-Osserman (\cite{LO}) and Harvey-Lawson (\cite{HL}).

\begin{example}[Lawson-Osserman Cone] \label{LOcone}
Let $\eta: \mathbb{S}^3 \rightarrow \mathbb{S}^2$ be the Hopf map:
\begin{equation*}
    \eta(z_1,z_2) = (2\overline{z_1}z_2, |z_1|^2 - |z_2|^2)
\end{equation*}
for $(z_1,z_2) \in \mathbb{C}^2$ with $|z_1|^2 + |z_2|^2 = 1$. Making the identifications $\mathbb{C}^2 \simeq \mathbb{R}^4$ and $\mathbb{C} \times \mathbb{R}\simeq \mathbb{R}^3$, we see that the graph of $\frac{\sqrt{5}}{2}|x|\eta\big(\frac{x}{|x|}\big)$ is a four dimensional minimal cone in $\mathbb{R}^7$ with an isolated singularity, called the \emph{Lawson-Osserman cone}. It was the first example of a singular Lipschitz minimal graph (\cite{LO}). This is in stark contrast to the codimension one case, in which all Lipschitz minimal graphs are smooth (\cite{G}). In \cite{HL}, it was shown that the Lawson-Osserman cone is coassociative.
\end{example}

A key fact is that the normal bundle of a coassociative submanifold is isometrically isomorphic to the space of anti-self dual $2$-forms, $\Lambda_{-}$, via the map $NM \rightarrow \Lambda_-(M)$ defined by $V \mapsto (\iota_V \omega_0)\lvert_M$, where $\iota_V$ denotes interior multiplication by $V$ (see \cite{M}). Given $V \in NM$, we will denote $\alpha_V := (\iota_V \omega_0)\lvert_M$ to be its associated anti-self dual $2$-form. Since the proof of the theorem in this section is similar to those of the previous section, we omit the repetitive details.

\begin{lemma}\label{neg1}
    Let $C$ be a coassociative cone in $\mathbb{R}^{7}$. Then there is no closed anti-self dual 2-form on $C$ with homogeneity $-1$. 
\end{lemma}
\begin{proof}
    If $\alpha$ is a $2$-form on $C$ with homogeneity $-1$, we can write $\alpha=dr \wedge \eta + r \omega$ for some $1$-form $\eta$ on $\Sigma$ and a 2-form $\omega$ on $\Sigma$. Then, we have
    \begin{align*}
        d\alpha &= -dr\wedge d_{\Sigma} \eta + dr \wedge \omega + r d_{\Sigma} \omega, \\
        *\alpha &= dr \wedge *_{\Sigma} \omega + r*_{\Sigma} \eta.
    \end{align*}        
        From the anti-self duality, we have $*_{\Sigma}\eta=-\omega$. Since $\alpha$ is closed, we have $\omega = d_{\Sigma} \eta$. Thus, $d_{\Sigma} \eta = -*_{\Sigma} \eta$.
    
    Now, we compute the Hodge Laplacian of the 1-form $\eta$ on $\Sigma$
    \begin{align*}
        \Delta_{\Sigma} \eta &= d\delta \eta + \delta d \eta \\
                             &= -d*d* \eta -*d*d \eta \\
                             &= -d*d(-d \eta) -*d*(-* \eta) \\
                             &= *d \eta \\
                             &= -\eta.
    \end{align*}
    This contradicts that the spectrum of the Hodge Laplacian is nonnegative.
    \end{proof}

Using Lemma \ref{neg1}, we can argue as in Theorem \ref{Lthm} to prove the strict stability of coassociative cones.    
\begin{theorem}\label{coassoc}
    Every coassociative cone in $\mathbb{R}^{7}$ is strictly stable.
\end{theorem}
\begin{proof}
    Suppose that $C$ is not strictly stable. Then Proposition \ref{strict} gives the existence of a Jacobi field $V$ of homogeneity $-1$ satisfying $|V|^2 \leq Kr^{-2}$ for some $K > 0$. Arguing as in Theorem \ref{Lthm}, using the second variation formula for coassociative submanifolds instead (\cite{M}, Theorem 4.9), we find that $d\alpha_V = 0$ contradicting Lemma \ref{neg1}.
\end{proof}

\bibliographystyle{amsplain}

\end{document}